\documentclass[reqno, 11pt]{amsart}
\usepackage{amsmath, amssymb,color, amsthm}
\usepackage{mathrsfs}
\usepackage[utf8]{inputenc}
\usepackage{stackrel}
\usepackage{geometry}
\numberwithin{equation}{section}

\usepackage{hyperref, xcolor}
\definecolor{goetheblau}{cmyk}{1.00 0.20 00 0.40}
\definecolor{hellgrau}{cmyk}{0.04 0.04 0.05 0.02}
\definecolor{sandgrau}{cmyk}{0.12 0.09 0.13 0}
\definecolor{dunkelgrau}{cmyk}{0.25 0.25 0.30 0.75}
\definecolor{purple}{cmyk}{0.08 1.00 0.30 0.36}
\definecolor{emorot}{cmyk}{0.04 1.00 0.80 0.07}
\definecolor{senfgelb}{cmyk}{0.01 0.25 1.00 0.05}
\definecolor{gruen}{cmyk}{0.62 0.40 0.87 0.09}
\definecolor{magenta}{cmyk}{0.08 0.86 0.12 0.12}
\definecolor{orange}{cmyk}{0 0.70 1.00 0.04}
\definecolor{sonnengelb}{cmyk}{0 0.12 0.95 0}
\definecolor{hellesgruen}{cmyk}{0.40 0.17 0.81 0.07}
\definecolor{lichtblau}{cmyk}{0.80 00 0.06 0.04}

\hypersetup{
	colorlinks = true,
	linkcolor  = emorot,
	citecolor = hellesgruen
}
\newtheorem{theorem}{Theorem}[section]
\newtheorem{corollary}[theorem]{Corollary}

\newtheorem{lemma}[theorem]{Lemma}
\newtheorem{proposition}[theorem]{Proposition}
\newtheorem{remark}[theorem]{Remark}

\newcommand{\N}{\mathbb{N}}
\newcommand{\Z}{\mathbb{Z}}
\newcommand{\R}{\mathbb{R}}

\usepackage{commath}
\usepackage{esdiff}
\DeclareMathOperator{\Var}{\mathbf{Var}}
\DeclareMathOperator{\EW}{\mathbf{E}}
\DeclareMathOperator{\WS}{\mathbf{P}}
\DeclareMathOperator{\Cov}{\mathbf{Cov}}

\usepackage{framed, todonotes, subcaption}

\usepackage{enumitem}

\allowdisplaybreaks

\title[Asymptotic Gaussianity  in the Hammond-Sheffield urn]{Asymptotic Gaussianity via coalescence probabilities in the Hammond-Sheffield urn}
\author[J.~L. Igelbrink]{Jan Lukas Igelbrink}
\address{Jan Lukas Igelbrink  \\ Institut f\"ur Mathematik, 
  Johannes Gutenberg-Universit\"at Mainz and Goethe-Universit\"at Frankfurt, Germany.}
\email{jigelbri@uni-mainz.de}
\author[A. Wakolbinger]{Anton Wakolbinger}
\address{Anton Wakolbinger, Institut f\"ur Mathematik, Goethe-Universit\"at Frankfurt, Germany.}
\email{wakolbinger@math.uni-frankfurt.de}
\begin{document}

\begin{abstract}
 For the renormalised sums of the random $\pm 1$-colouring of the connected components \mbox{of $\mathbb Z$} generated
by the coalescing renewal processes in the ``power law P\'olya's urn'' of Hammond and Sheffield \cite{HS} we prove functional convergence towards fractional Brownian motion, closing a gap in the tightness argument of their paper.

In addition, in the regime of the strong renewal theorem we gain insights into the coalescing renewal processes in the Hammond-Sheffield urn (such as the asymptotic depth of most recent common ancestors) and are able to control the coalescence probabilities of two, three and four individuals that are randomly sampled from $[n]$. This allows us to obtain a new, conceptual proof of the asymptotic Gaussianity (including the functional convergence) of the renormalised sums of more general colourings, which can be seen as an invariance principle beyond the main result of  \cite{HS}.
 
In this proof, a key ingredient of independent interest is a sufficient criterion for the asymptotic Gaussianity of the renormalised sums in randomly coloured random partitions of $[n]$, based on Stein’s method.
 
 Along the way we also prove a statement on the asymptotics of the coalescence probabilities in the long-range seedbank model of Blath, Gonz\'alez Casanova, Kurt, and Span\`o, see \cite{BGKS}.
\end{abstract}
\maketitle
\noindent
\newpage
\tableofcontents
\section{Introduction}\label{secIntro}
\noindent
 We start with a brief description of the model of \cite{HS} and then state our main results together with a short outline of the paper.
 
For $0<\alpha < \tfrac12 $ and a slowly varying function $L: \R \rightarrow \R^+$ let $\mu:=\mu_{\alpha, L}$ be a probability measure on $\N=\{1,2,\ldots\}$ having the power law tails
\begin{equation}\label{basic}
 \mu\left(\left\{n,n+1,\ldots\right\}\right) \sim n^{-\alpha} L(n) \mbox { as } n\to \infty,
\end{equation}
with the usual convention that for two sequences $f(n)$, $g(n)$ of real numbers
\begin{equation*}
  f(n) \sim g(n) \mbox{ as }n\rightarrow \infty \end{equation*}means that $\lim_{n\to \infty} f(n)/g(n) = 1$.
Throughout it will be assumed that 
\begin{equation}\label{nonlattice}
 \mbox{the greatest common divisor of } \left\{n\in \N: \mu(n) > 0\right\} \mbox{ is one.}
 \end{equation}
Let $R$ be an $\mathbb N$-valued random variable with distribution $\mu$. A random directed graph $G_\mu$ with vertex set $\Z$ is generated in the following way: Let $\left(R_i\right)_{i\in\Z}$ be a family of independent copies of $R$. The random set of edges $E\left( \mathcal G_\mu \right)$ is then given by
\begin{equation*}
  E\left( \mathcal G_\mu \right):= \left\{ \left(i, i- R_i\right): i\in\Z \right\}.
\end{equation*}
This induces the random equivalence relation
\begin{equation}\label{simdef}
  i \sim j :\Longleftrightarrow i \mbox{ and }j \mbox { belong to the same connected component of } \mathcal G_\mu.
\end{equation}
Note that the symbol $\sim$  is used in \eqref{basic} and \eqref{simdef} in two different meanings; this will cause no risk of confusion.

 For $i \in \mathbb Z$ the  connected component containing  $i$ is denoted by $\mathscr C_i$. 
The random variables $(R_i)_{i\in \mathbb Z}$ give rise to {\em coalescing renewal processes} starting from the integers; see Section~\ref{secBGKS}  for an interpretation (and extension) in terms of the long-range seedbank model of \cite{BGKS}. In this terminology $\mathcal G_\mu$ is the graph of ancestral lineages of the individuals $i \in \mathbb Z$, and the component~$\mathscr C_i$ consists of all $j\in \mathbb Z$ that are related to $i$, see Figure~\ref{fig:randomgraph} for an illustration. 
\begin{figure}
  \centering
  \includegraphics[width=\textwidth]{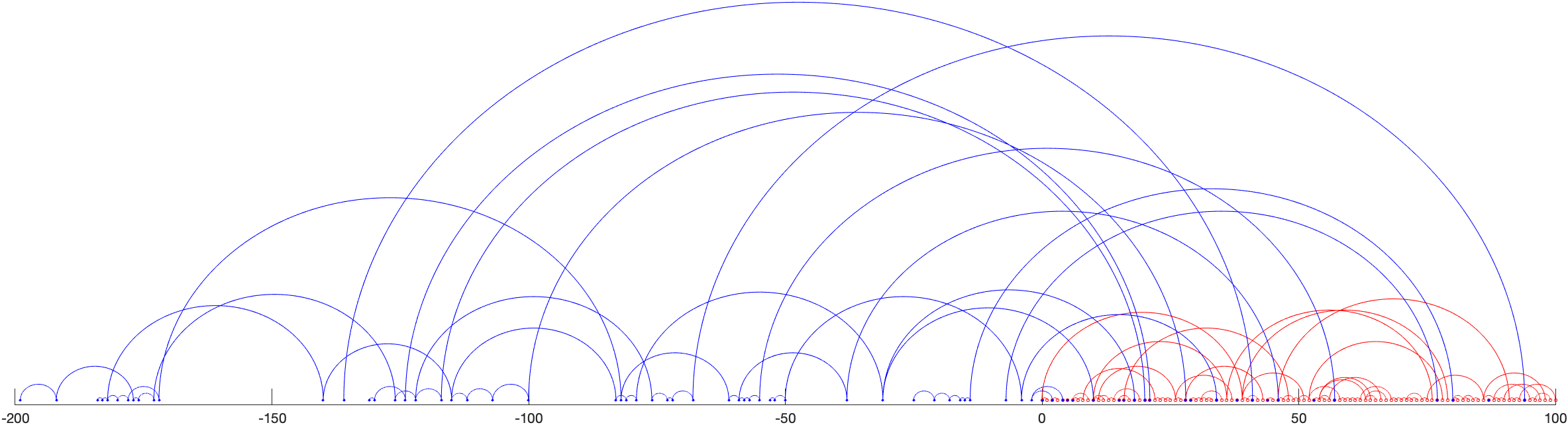}
  \caption{A realisation of the ancestral lineages of the individuals $\{0,\ldots, 100\}$ traced back till $-200$. Each of the arcs corresponds to an edge of $G_\mu$. All the outgoing edges  from $i=0,\ldots, 100$ which map to an individual in $\{0,\ldots, 100\}$ are drawn (in red), whereas for $i$ between $-200$ and $-1$ only those outgoing edges are drawn (in blue) that belong to an ancestral lineage of some $j \in \{0,\ldots, 100\}$.
 Here the exponent $\alpha$ in \eqref{basic} was chosen as $0.39$.
  }
  \label{fig:randomgraph}
\end{figure}
The probability that $0$ belongs to the ancestral lineage of $n$ is thus given by the weight assigned to $n$ by the renewal measure,
\begin{equation}\label{defqn}
 q_n := \WS \left( \tilde R_1+\ldots+ \tilde R_j =n \mbox{ for some }j\geq 0 \right)
\end{equation}
with $\tilde R_1, \tilde R_2,\ldots$ being  independent copies of $R$. (Note that $\WS(0 \sim n)$ is in general larger than $q_n$ because $0$ and $n$ may be related to each other even if $0$ is not an ancestor of $n$.)
 
Hammond and Sheffield suggest the picture of an urn in which the types of the individuals $i$ are determined recursively: each individual~$i \in \mathbb Z$ inherits the type (or ``colour'') of its parent $i-R_i$. With $\{+1, -1\}$ as  the set of colours, they show that the set of random colourings of $\mathbb Z$ that are consistent with $\mathcal G_\mu$ has a Gibbs structure, with the extremal elements being given by i.i.d.~assignments of colours to the connected components of $\mathcal G_\mu$. The main result of \cite{HS} concerns the asymptotics of the rescaled sum  over the types of the individuals $1,\ldots, \lfloor tn \rfloor $, $t\ge 0$, which as $n \to \infty$ turns out to converge to fractional Brownian motion. The individuals' types arise as follows:

Assume that each component of $\mathcal G_\mu$ gets coloured by an independent copy of a real-valued random variable $Y$.
In the situation of  \cite{HS}, $Y$ is a centered Rademacher$(p)$ variable, i.e.
\begin{equation}
\label{Rade} Y = \xi-(2p-1) \mbox{ with } \WS(\xi=+1)=p, \,\WS(\xi=-1)=1-p.
\end{equation}
For $i\in \mathbb Z$ the colour of the component  $\mathscr C_i$ will be denoted by $Y_i$. 
Define the ``random walk'' (with dependent increments)
\vspace{-0.3cm}
\begin{equation}\label{eq:Sndef}
  S_n := \sum_{i=1}^n Y_i,  \quad n=0,1,\ldots.
\end{equation}
\vspace{-0.1cm}
By construction, 
\begin{equation} \label{sigman}\sigma_n^2 := \Var[S_n] = \sum_{i,j \in [n]} \Cov[Y_i,Y_j] = \EW[Y^2] \, \sum_{i,j \in [n]} \WS(i \sim j).
\end{equation}
\cite[Lemma 3.1]{HS} show by Fourier and Tauberian arguments that 
\begin{equation} \label{asvar}
\sum\limits_{i,j \in [n]} \WS(i\sim  j) \sim \frac{C_\alpha}{\alpha(2\alpha+1)} \frac{n^{2\alpha +1}}{L(n)^2} \quad \mbox{ as } n\to \infty,
\end{equation}
with 
\vspace{-0.2cm}
\begin{equation}\label{Calpha}
     C_\alpha := \frac 1{\sum_{m\ge 0} q_m^2} \frac{\Gamma(1-2\alpha)}{\Gamma(\alpha)\Gamma(1-\alpha)^3}.
     \end{equation}
     We will obtain \eqref{asvar} as a corollary of Proposition \ref{probrel} below, which requires the additional condition
     \begin{equation}\label{qas}
  q_n \sim \frac{1}{\Gamma(\alpha)\Gamma(1-\alpha)} \frac{n^{\alpha -1}}{L(n)} \quad \mbox{as } n\rightarrow \infty.
\end{equation}
This condition, which also appears in our Theorem \ref{mainth}, is equivalent to the validity of the Strong Renewal Theorem for the renewal process with an increment distribution $\mu$ satisfying  \eqref{basic} and \eqref{nonlattice}, see  \cite{CaravennaDoney}, whose Theorem~1.4 gives necessary and sufficient conditions in terms of $\mu$ for the validity of \eqref{qas}. A well-known sufficient condition for \eqref{qas} is the criterion of \cite{D} \begin{equation} \label{Doney}
 \sup_{n \ge  1} \frac{n\WS(R = n)}{\WS(R > n)}<\infty.
\end{equation}

For $\frac {i-1}n \le t \le \frac{i}n$, $i,n \in \mathbb N$,  let $S^{(n)}(t)$ be the linear interpolation of $S_i/\sigma_n$ and $S_{i+1}/\sigma_n$.
 Because of \eqref{sigman} and \eqref{asvar}, for all $t\ge 0$,
$$\Var\left[S^{(n)}(t)\right] \to t^{2\alpha +1} \quad \mbox {as } n\to \infty.$$
Since $(S_n)_{n\in \N_0}$ has stationary increments by construction, this implies the convergence
\begin{equation*}\Cov\left(S_s^{(n)}, S_t^{(n)}\right) \stackrel[n\to \infty]{}\longrightarrow \frac 12\left(  s^{2\alpha +1} + t^{2\alpha +1} -  |t-s|^{2\alpha+1} \right), \quad s,t \ge 0.\end{equation*}
The right-hand side is the covariance function of {\em fractional Brownian motion with Hurst parameter \mbox{$H=\tfrac12+\alpha$}}, which is the unique centered Gaussian process with  variance function $t^{2H}$, $t\ge 0$, stationary increments and a.s.~continuous paths.
The processes $S^{(n)}$ are centered as well. Thus, in order to prove that $S^{(n)}$ converges as $n\to \infty$ (in the sense of finite dimensional distributions) to 
fractional Brownian motion with Hurst parameter $H$, it only remains to show that the finite dimensional distributions of~$S^{(n)}$ are asymptotically Gaussian. This is provided by 
\begin{theorem}\label{mainth} Let $\mu$ be a probability measure on $\N$ satisfying \eqref{basic} and \eqref{nonlattice}. Assume  one of the following conditions \ref{theoremmain:A} or \ref{theoremmain:B}:
  \begin{enumerate}[label=(\Alph*)]
\item The colouring $Y$ is given by ~\eqref{Rade}.\label{theoremmain:A} 
\item The weights $q_n$ of the renewal measure specified in \eqref{defqn} satisfy  the asymptotics 
\eqref{qas},
and the  
colouring $Y$ obeys
\begin{equation}\label{fourth}
\EW\left[Y\right] = 0 \mbox{ and } 0< \EW\left[Y^4\right] < \infty. \end{equation} \label{theoremmain:B}
\end{enumerate}
\vspace{-0.3cm}
Then,
for any fixed  $d \in \mathbb N$ and  fixed  $0< t_1 < \cdots < t_{d}^{}<\infty$,  the sequence $\left(S_{\lfloor t_1 n\rfloor}, \ldots, S_{\lfloor t_{d} n\rfloor }\right)_{n\in \N}$ is asymptotically Gaussian as $n \to \infty$.
\end{theorem}
Under assumption \ref{theoremmain:A} of Theorem \ref{mainth},  for each fixed $t > 0$ asymptotic Gaussianity of $S_{\lfloor t n\rfloor}$ as $n\to \infty$  is proved in \cite{HS} via a martingale central limit theorem. The computations which ensure the applicability of the martingale CLT are quite subtle and involved,  making substantial use of the specific form \eqref{Rade} of the colouring of the random graph $\mathcal G_\mu$. In  \cite{HS} it is not explicitly discussed whether these arguments also carry over to the joint asymptotic Gaussianity of  $S_{\lfloor t_1 n\rfloor}, \ldots, S_{\lfloor t_m n\rfloor }$ for fixed $t_1 < \cdots < t_m$. However, by applying the martingale CLT to linear combinations of these random variables one can check  that this is indeed the case.

Under assumption \ref{theoremmain:B} we give a new, conceptual proof of the asymptotic Gaussianity of the finite dimensional distributions of  $S^{(n)}$. This proof, which is completed in Section~\ref{secmainproof}, is based on  insights into the structure of $\mathcal G_\mu$ which are stated in Section~\ref{secHScoal} and proved in Sections~\ref{Proofprobrel}-\ref{ProofCovest}. A key ingredient in the new proof is Theorem~\ref{asGausscolored}, which provides a criterion for the asymptotic Gaussianity in randomly coloured random partitions also in a more general setting.   Proposition \ref{propYnew}, which is instrumental in the proof of Theorem~\ref{asGausscolored}, is based on Stein's method and yields the closeness of the distribution of $S_n/\sigma_n$ to the standard normal distribution in terms of a bound  that involves $\Var \left[Y^2\right]$; this explains the finiteness condition of $\EW\left[Y^4\right]$ in \eqref{fourth}. 

Let us also mention that the loss of ground which comes with assuming the ``strong renewal'' condition \eqref{qas}  in addition to~\eqref{basic} and \eqref{nonlattice} seems rather minor. 
Indeed it becomes clear from the examples in \cite[Section~10]{CaravennaDoney} that the class of measures $\mu$ which satisfy~\eqref{basic} and \eqref{nonlattice} but fail to satisfy \eqref{qas} is rather special.

On the other hand, the benefit of assuming \eqref{qas} is  twofold. Firstly, it allows a direct analysis of asymptotic properties of the genealogy of the coalescing renewal processes in the Hammond-Sheffield urn, see Propositions \ref{probrel},  \ref{3and4} and \ref{propMRCA}. Secondly, this opens the way to a two-step analysis (first of the random partition of $\mathbb Z$, then of its random colouring) which allows to derive the ``invariance principle'' stated in Theorem \ref{mainth}\ref{theoremmain:B}.  
\\\\
\indent The following implication of Theorem \ref{mainth} is immediate from its introductory discussion.
\begin{corollary}\label{findim} Under assumptions \ref{theoremmain:A} or \ref{theoremmain:B}, 
$S^{(n)}$ converges as $n\to \infty$  in the sense of finite dimensional distributions to  fractional Brownian motion with Hurst parameter $H=\tfrac12+\alpha$.
\end{corollary}

The next result, which will be proved in Section \ref{tight}, amends the proof of \cite[Lemma 4.1]{HS}, see Remark \ref{HSLemma}. Here, for each $n \in \mathbb N$ and $T>0$,   $\left(S^{(n)}(t)\right)_{0\le t\le T}$ is viewed as a random variable taking its values in $C\left([0,T]; \mathbb R\right)$, the space of continuous functions from $[0,T]$ to~$\mathbb R$, equipped with the $\sup$-norm.
\begin{proposition}\label{proptight} Under the assumptions of Theorem \ref{mainth},
for all $T>0$ the sequence of random variables $\left(S^{(n)}(t)\right)_{0\le t\le T}$ is tight.
\end{proposition}
A direct consequence of Corollary \ref{findim} and Proposition~\ref{proptight} is
\begin{corollary}\label{funcconv} Under the assumptions of Theorem \ref{mainth},
$S^{(n)}$ converges in distribution (with respect to the topology of locally uniform convergence) to  fractional Brownian motion with Hurst parameter $H=\tfrac12+\alpha$.
\end{corollary}
\section{Coalescence probabilities in the Hammond-Sheffield urn}\label{secHScoal}
In this section we will assume that the weights $q_n$ of the renewal measure defined in \eqref{defqn} obey the asymptotics
\eqref{qas}, see the discussion of this condition in Section \ref{secIntro}.
\begin{proposition}\label{probrel} The coalescence probabilities for the ancestral lineages obey the asymptotics
  \begin{equation}\label{assim}
   \WS(0\sim i)\sim C_\alpha \frac {i^{2\alpha -1}}{L(i)^2} \quad \mbox{ as  } i\to \infty,
     \end{equation}
with $C_\alpha$ as in \eqref{Calpha}.
        \end{proposition}
 \begin{remark}\label{consequence}
   \begin{enumerate}[label=(\alph*)]
   \item The asymptotics \eqref{asvar} is a direct consequence of \eqref{assim}. Indeed, the latter implies 
     \begin{equation} \label{newoneseven}
     \sum_{i\in [n]} (n-i)\WS(i\sim 0) \sim  n^{2\alpha +1} \frac{C_\alpha}{L(n)^2}\frac 1n \sum_{i=1}^n \left(1-\frac in\right)\left( \frac in \right)^{2\alpha-1} \quad \mbox{as } n\to \infty, 
     \end{equation}
with the limit of the Riemann sums being 
$$\int_0^1 (1-x) x^{2\alpha -1} \mathrm{d}x = \frac 1{2\alpha (2\alpha+1)}.$$
Since the left-hand sides of  \eqref{newoneseven} and \eqref{asvar} are equal, this shows the asserted implication.\\
\item\label{remark:consequenceb}  In the light of the proof of Proposition \ref{probrel} (carried out in Section~\ref{Proofprobrel}) we conjecture that increment distributions $\mu$ that satisfy \eqref{basic} {\em and} violate \eqref{qas}, generically also do not admit the asymptotics~\eqref{assim}. \end{enumerate}
   \end{remark}
   The next result, Proposition \ref{3and4}, will be instrumental in the proof of Theorem \ref{mainth} under Assumption \ref{theoremmain:B}.
This proposition will consider  the probability that three (respectively four) individuals that are randomly chosen from $[n]$ belong to the same component of $\mathcal G_\mu$.
      \begin{proposition}\label{3and4} Let
   $\mathscr I^{(n)} , \mathscr J^{(n)}, \mathscr K^{(n)}$ and \mbox{$\mathscr L^{(n)}$}
    be independent and  uniformly distributed on~$[n]$, and independent of the random graph $\mathcal G_\mu$. Then for all $\delta>0$, as $n \to \infty$, 
 \begin{equation}
 \label{trip}\phantom{AAA}\WS\left(\mathscr I^{(n)} \sim \mathscr J^{(n)} \sim  \mathscr K^{(n)}\right)  = O\left(n^{4\alpha-2+\delta}\right), 
 \end{equation}
 \begin{equation}
\label{quadr}\WS\left(\mathscr I^{(n)} \sim \mathscr J^{(n)} \sim  \mathscr K^{(n)}  \sim  \mathscr L^{(n)}\right) = O\left(n^{6\alpha-3+\delta}\right).
\end{equation}
\end{proposition}
Proposition \ref{3and4} will be proved in Section \ref{Proofthreefour}. The following corollary will bound the triplet and quartet coalescence probabilities addressed in \eqref{trip} and \eqref{quadr} asymptotically as $n\to \infty$ by powers of the pair coalescence probability. Estimates of this kind will be required in Theorem  \ref{defmultiS}; note that for sufficiently small $\varepsilon$ the powers guaranteed by Corollary \ref{cor3and4} are strictly larger than those required in Theorem  \ref{defmultiS}.

      \begin{corollary}\label{cor3and4} Let
   $\mathscr I^{(n)} , \mathscr J^{(n)}, \mathscr K^{(n)}$ and \mbox{$\mathscr L^{(n)}$}
    be as in Proposition \eqref{3and4}.Then for all $\varepsilon>0$, as $n \to \infty$, 
 \begin{equation}
\label{trip1}\phantom{AAA}\WS\left(\mathscr I^{(n)} \sim \mathscr J^{(n)} \sim  \mathscr K^{(n)}\right)   =O\left(\left(\WS\left(\mathscr I^{(n)} \sim \mathscr J^{(n)}\right)\right)^{2-\varepsilon}\right), 
 \end{equation}
 \begin{equation}
\label{quadr1}\WS\left(\mathscr I^{(n)} \sim \mathscr J^{(n)} \sim  \mathscr K^{(n)}  \sim  \mathscr L^{(n)}\right) = O\left(\left(\WS\left(\mathscr I^{(n)} \sim \mathscr J^{(n)}\right)\right)^{3-\varepsilon}\right).
\end{equation}
\end{corollary}
The proof of Corollary \ref{cor3and4} is immediate from Proposition \ref{3and4} together with \eqref{asvar}. Indeed, \eqref{asvar} asserts that the order of $\WS\left(\mathscr I^{(n)} \sim \mathscr J^{(n)}\right)$ is $\frac {n^{2\alpha-1}}{L(n)^2}$ as $n\to \infty$. With $\varepsilon$ prescribed as in Corollary~\ref{cor3and4}, it thus suffices to choose in Proposition \ref{3and4} a positive $\delta$ that is smaller than $\varepsilon\left(1-2\alpha\right)$.

 Although the next result, Proposition \ref{propMRCA}, will not be used explicitly in the proof of Theorem \ref{mainth}, it seems interesting in its own right and also gives an intuition why the  estimates  in Corollary~\ref{cor3and4} should hold. Qualitatively,  Proposition \ref {propMRCA} says that for large $n$ the ancestral lineages of $0$ and $n$ with high probability either coalesce quickly (i.e. \emph{on the scale~$n$}) or never. This makes it believable that, as asserted in Corollary~\ref{cor3and4},  the triplet coalescence probability  should asymptotically be comparable to the square of the  pair coalescence probability, and that the quartet coalescence probability should roughly be equal to the third power of the pair coalescence probability. 
    \begin{proposition}\label{propMRCA}
    Let $\mathscr M(0,n) := \max\left\{j \le 0: j \sim 0 \mbox{ and }  j \sim n\right\}$ (with $\max \emptyset := -\infty$) be the most recent common ancestor of\,  $0$ and $n$, and put $D_n := -\mathscr M(0,n)$. Then,  as $n \to \infty$,  the sequence of random variables $\frac {D_n}n$, conditioned under $\{0 \sim n\}$, converges in distribution to the random variable $D$ with density $B(\alpha, 1-2\alpha)^{-1} x^{\alpha-1}(1+x)^{\alpha-1}\, dx$, $x > 0$.   
  \end{proposition}
  Proposition \ref{propMRCA} will be proved in Section~\ref{ProofMRCA}. The distribution of the random variable $D$ appearing in Proposition \ref{propMRCA} is known as Beta prime distribution with parameters $\alpha$ and $1- 2\alpha$; it arises as the distribution of $B/(1-B)$ where $B$ is Beta($\alpha, 1- 2\alpha$) distributed. 
  See Figure~\ref{fig:mrca} for simulations of the ancestral lineages, which also illustrate the depths of the most recent common ancestors. 
  \begin{figure}
  \begin{subfigure}{0.33\textwidth}\centering
    \includegraphics[height=3.5\textwidth]{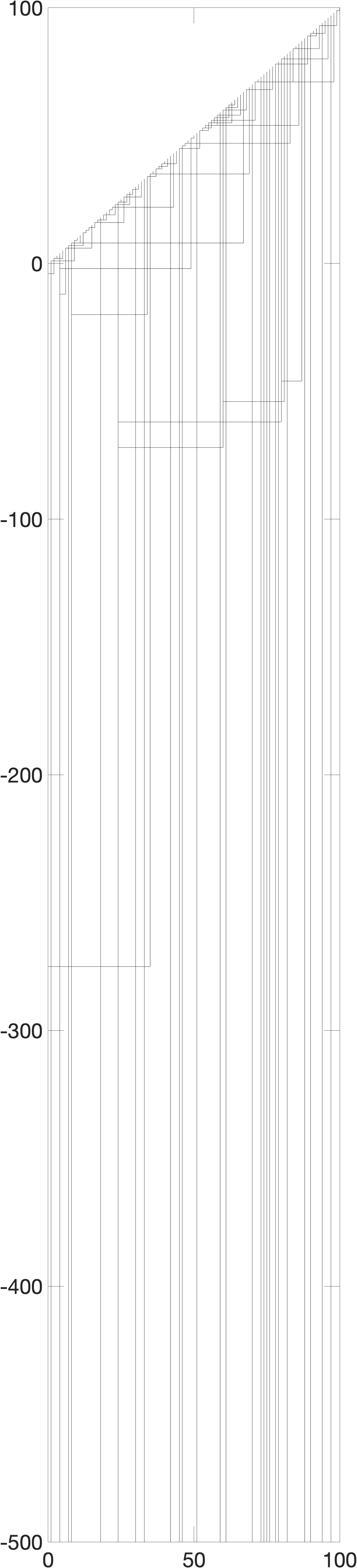}
    \caption{}\label{subfigure:1}
  \end{subfigure}
  \begin{subfigure}{0.33\textwidth}\centering
    \includegraphics[height=3.5\textwidth]{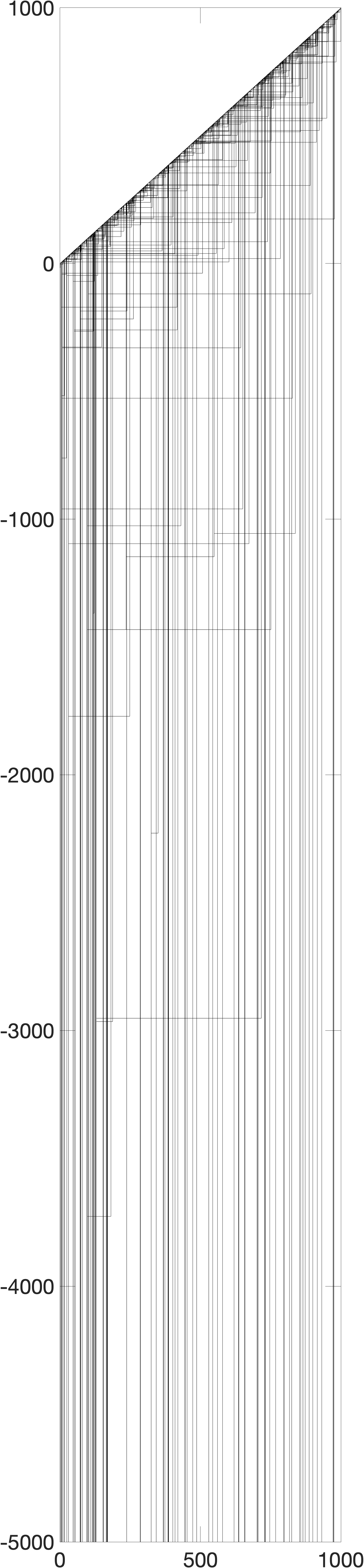}
    \caption{}
    \label{subfigure:2}
  \end{subfigure}
  \caption{This is a simulation of the ancestral lineages of the individuals $0,\ldots, n$,
  with $n = 100$ in panel~(\subref{subfigure:1}) and $n = 1000$ in panel~(\subref{subfigure:2}). Direction of time is vertical, and horizontal lines mark coalescence events. The two panels give an impression of how the genealogical forest of the individuals $i \in [n]$ scales with $n$, see e.g. Proposition~\ref{propMRCA}. Like in Figure \ref{fig:randomgraph}, the parameter $\alpha$ was chosen as 0.39.}
  \label{fig:mrca}
\end{figure}

The following lemma will also be important in the proof of Theorem~\ref{mainth} under Assumption \ref{theoremmain:B}.
\begin{lemma} \label{Covest} Let $\sim$ be the random equivalence relation defined in \eqref{simdef}. For $i,j,k,\ell \in \mathbb Z$, 
    \begin{equation} \label{Covineq}
    \Cov\left[I_{\{i\sim j\}}, I_{\{k\sim \ell\}}\right] \le \WS\left(i\sim j \sim k \sim \ell\right).
    \end{equation}
      \end{lemma}
       \noindent  Here and below, $I_E$ denotes the indicator variable of an event $E$. 
\begin{remark}
The proof of Lemma \ref{Covest} (given in Section \ref{ProofCovest}) shows that \eqref{Covineq} holds for general increment distributions $\mu$, without the assumptions \eqref{basic} and \eqref{nonlattice}.
\end{remark} 
\section{Asymptotic Gaussianity in randomly coloured random partitions}\label{secStein}
In this section we consider a situation that is more general than the one described in Section~\ref{secIntro}. For $m \in \mathbb N$ let $\mathscr P^{(m)}$ be a random partition of $[m]$. The (random) equivalence relation on $[m]$ induced by $\mathscr P^{(m)}$ will be denoted  by  $\stackrel m\sim $, i.e. 
\begin{equation}\label{simmdef}
  i \stackrel m \sim j :\Longleftrightarrow i \mbox{ and }j \mbox { belong to the same partition element of } \mathscr P^{(m)}.
\end{equation}
The situation described in Section~\ref{secIntro} fits into this framework, by choosing $\stackrel m \sim$ as the restriction to the set $[m]$ of the equivalence relation $\sim$ defined in \eqref{simdef}. Note, however, that this kind of consistency of the relations  $\stackrel m \sim$ is not required in the present section.

Let $Y$ be a real valued random variable  with $\EW[Y]~=~0$ and $0<\EW[Y^4] <\infty$. Thinking of each partition element being ``coloured''  by an independent copy of~$Y$,
we write $Y_i^{(m)}$ for the colour of the partition element in $\mathscr P^{(m)}$ to which $i \in [m]$ belongs.   We then define for $k \in \N$
$$Z_k^{(m)}:= \sum_{i=1}^kY^{(m)}_i\, .$$ 
In the sequel we fix a natural number  $d$ and real numbers $0=\rho_0 < \rho_1 < \cdots < \rho_{d}=1$.
The following theorem presents a sufficient criterion for the asymptotic normality of the sequence of $\mathbb R^{d}$-valued random variables 
\begin{equation}\label{defmultiS}
\mathcal Z^{(m)} := \left(Z^{(m)}_{\lfloor \rho_1^{}m\rfloor}, \ldots, Z^{(m)}_{\lfloor \rho_{d}^{}m\rfloor}\right)
\end{equation} 
as $m\to \infty$.
To prepare for this, let for all $m \in \N$ the random variables
   $\mathscr I^{(m)} , \mathscr J^{(m)}, \mathscr K^{(m)}$ and \mbox{$\mathscr L^{(m)}$}
    be independent and  uniformly distributed on~$[m]$, and independent of $\mathscr P^{(m)}$ and of~$\left(Y^{(m)}_i\right)_{i\in [m]}$.
     \begin{theorem}\label{asGausscolored}
The sequence of $\R^{d}$-valued random variables $\mathcal Z^{(m)}$ defined in \eqref{defmultiS} is asymptotically Gaussian as $m\to \infty$ provided the following  conditions are satisfied:
 \begin{eqnarray} \label{drei}
 \WS\left(\mathscr I^{(m)}  \stackrel m\sim \mathscr J^{(m)}  \stackrel m\sim  \mathscr K^{(m)}\right) \qquad &=& o\left(\left(\WS\left(\mathscr I^{(m)}  \stackrel m \sim \mathscr J^{(m)} \right)\right)^{3/2}\right)  \mbox{ as } m\to \infty, \\ \label{vier}
\WS\left(\mathscr I^{(m)}  \stackrel m\sim \mathscr J^{(m)}  \stackrel m\sim  \mathscr K^{(m)}  \stackrel m\sim \mathscr L^{(m)}\right) &=& o\left(\left(\WS\left(\mathscr I^{(m)}  \stackrel m \sim \mathscr J^{(m)} \right)\right)^2\right) \quad\mbox{ as } m\to \infty,
\end{eqnarray}
\begin{equation} \label{Covineqn}
    \Cov\left[I_{\{i\stackrel m\sim j\}}, I_{\{k\stackrel m\sim \ell\}}\right] \le \WS\left(i\stackrel m \sim j \stackrel m\sim k \stackrel m\sim \ell\right) \quad \mbox{\rm  for all } m \in \mathbb N \mbox{ \rm  and } i,j,k,\ell \in [m],
    \end{equation}
    and for all  $\left(\alpha_1, \ldots, \alpha_{d}\right) \in \R^{d}\setminus \{(0,\ldots, 0)\}$ and
    $$a_i^{(m)} := \alpha_g \mbox{ \rm  if } \lfloor \rho_{g-1}^{} m\rfloor  < i \le \lfloor \rho_{g} m\rfloor,    \quad   i=1,\ldots m;   \quad g=1,\ldots, d$$ there exists a constant $\widetilde C>0$ (not depending on $m$) such that 
    \begin{equation} \label{variancecomp}
\sum_{i,j=1}^m a_i^{(m)} a_j^{(m)} \WS\left(i \stackrel m \sim j\right)\ge \widetilde C \sum_{i,j=1}^m \WS\left(i \stackrel m \sim j\right), \quad m\in\N.
\end{equation}
\end{theorem}  
\begin{remark}\label{remark:th31}
  \begin{enumerate}[label=(\alph*)]
  \item \label{remark:th31a} Asymptotic Gaussianity of the univariate sequence $(Z_m^{(m)})_{m\in \N}$ is implied by the first three of the four conditions in Theorem \ref{asGausscolored}. Indeed, for $d=1$, condition \eqref{variancecomp} is automatically satisfied with $\widetilde C := \alpha_1$, since in that case $a_i^{(m)} =  \alpha_1$ for $i=1,\ldots, m$.\\
    \item \label{remark:th31b} If the  partitions $\mathscr P^{(m)}$ are induced by a Hammond-Sheffield urn as described in Section \ref{secIntro}, then {\em all} conditions of Theorem \ref{asGausscolored} (including  condition \eqref{variancecomp} for {\em all} $d \in \N$) are implied by the assumption~\eqref{qas}, see Section \ref{secmainproof}.
  \end{enumerate}
\end{remark}
Theorem \ref{asGausscolored} will be proved using the following proposition which, in turn, will be deduced from a theorem of Charles Stein, see  \cite[Lecture~X, Theorem 1]{stein}.  Since $m$ will be fixed in this proposition, we will write $\sim $ instead of $\stackrel m \sim$ for notational convenience and without any risk of confusion.
\begin{proposition}\label{propYnew} For fixed $m \in \N$ and $a_1, \ldots , a_m \in \R$ let
\begin{equation}\label{defSn}
  \bar S_m := \sum_{i=1}^m a_iY_i, \quad \bar \sigma_m^2 := \Var[\bar S_m],
\end{equation}
with $Y_i := Y_i^{(m)}$ as defined at the beginning of this section. 
Let $\mathcal N$ be a standard normal random variable. Then
for all continuously differentiable functions \mbox{$h:\R\rightarrow \R$}  with compact support
\begin{eqnarray}\label{newY}
\begin{split}
&\left\vert \EW\left[ h\left( \frac{\bar S_m}{\bar \sigma_m}\right) \right]- \EW [h(\mathcal N)]  \right\vert \\ 
&\leq \frac {c_1(h)}{\bar \sigma_m^2}\sqrt{\Var \left[ \sum\limits_{i,j=1}^m Y_i^2a_ia_j I_{\{i\sim j\}}\right] }
                                                                                                                                         +  \frac{c_2(h)}{\bar \sigma_m^3 }\EW\left[ \sum\limits_{i,j,k=1}^m |Y_i|^3 |a_i|a_ja_kI_{\{i\sim j\sim k\}} \right].
\end{split}
  \end{eqnarray}
  where the finite numbers $c_1(h)$ and $c_2(h)$ are defined as 
$$ c_1(h):= 2 \sup|h-\EW [h(\mathcal N)]|, \quad c_2(h):= 2\sup|h'|.$$
  \end{proposition}
\begin{proof}
For $i \in [m]$ we put 
\begin{equation}\label{defXi}
X_i := \frac {a_iY_i}{\bar \sigma_m}, \qquad M_i : = \{j\in [m]: j\sim i\},
\end{equation}
i.e., $M_i$ is that element of the partition $\mathscr P^{(m)}$ which contains $i$. With $\mathscr I$ being a uniform pick from~$[m]$ that is independent of $\mathscr P^{(m)}$, 
we write
\[M:= (M_1, \ldots, M_m), \quad W:= \sum_{i=1}^m X_i,  \quad W^\ast := W-\sum_{j\in M_{\mathscr I}}X_j, \quad G:= mX_{\mathscr I},\]
\[\mathcal B := \sigma\left(M, X_1,\ldots, X_m\right), \qquad \mathcal C:= \sigma\left(M,\mathscr I, (X_j)_{j\nsim \mathscr I}\right) \]
and note that $W= \EW[G|\mathcal B]$ a.s. Now \cite[Lecture~X, Theorem 1]{stein} asserts that
\begin{eqnarray}
  &&\left\vert \EW\left[ h\left(W\right) \right]- \EW [h(\mathcal N)]  \right\vert\nonumber\\
 & \leq&  c_1(h)\left(\sqrt{\EW \left[\left(1-\EW\left[G(W-W^\ast)| \mathcal B \right]\right)^2\right] }+c\EW\left[\left\vert\EW\left[G\left\vert\mathcal C\right.\right]\right\vert\right]\right)\label{twosummandsStein}\label{Steinpur} \\&&+
 c_2(h)\EW\left[|G|(W-W^\ast)^2\right], \nonumber
\end{eqnarray}
with $N$, $c_1(h)$ and $c_2(h)$ as in \eqref{newY} and a constant $c>0$. Let us first turn to the term under the square root on the right-hand side of \eqref{Steinpur} and observe that
\begin{eqnarray*}
 \EW \left[ G (W-W^*)| \mathcal B \right]= \EW\left[\left. mX_{\mathscr I} \sum_{j\in M_{\mathscr I}} X_j\, \right\vert\, \mathcal B \right] = \sum_{i=1}^m X_i\sum_{j\in M_i}X_j = \frac 1{\bar \sigma_m^2} \sum_{i=1}^m a_iY_i\sum_{j\in M_i}a_jY_j.
   \end{eqnarray*}
  The expectation of this random variable is 1, since 
  $$
  \EW\left[\EW\left[\frac 1{\bar \sigma_m^2} \sum_{i=1}^ma_iY_i\sum_{j=1}^m a_jY_j \bigg \vert \mathscr P^{(m)} \right]\right] = \frac 1{\bar \sigma_m^2} \Var[\bar S_m] =1.$$
 Hence the term under the square root in \eqref{Steinpur} equals $$\Var \left[\sum_{i=1}^m X_i \sum_{j\in M_i}X_j\right].$$ 
 
 The term $\EW\left[\left\vert\EW\left[G\left\vert\mathcal C\right.\right]\right\vert\right]$  in the right-hand side of  \eqref{Steinpur} vanishes, since the assumed independence of the colouring and the partitions  together with the assumption  $\EW[X_i]=0$ implies
 \begin{equation*}
   \EW\left[\left\vert \EW\left[G\left\vert\mathcal C\right.\right] \right\vert \right] =  \EW\left[\left\vert \EW\left[\left.m X_{\mathscr I} \right\vert M, \mathscr I\right] \right\vert\right] = 0.
 \end{equation*}
 Finally, the rightmost term in  \eqref{Steinpur} equals
 \begin{eqnarray*}
  &&\EW\left[ |G| \left( W-W^* \right)^2 \right]= \EW\left[ |G| \left( \sum_{j \in M_{\mathscr I}}X_j \right)^2  \right]\\
                                              &=& \sum_{i=1}^m \EW\left[I_{\{\mathscr I=i\}} |m X_i| \left( \sum_{j \in M_i}X_j \right)^2 \right]= \sum_{i=1}^m \EW\left[  | X_i| \left( \sum_{j \in M_i}X_j \right)^2 \right].                                      
\end{eqnarray*}
In summary we have shown that the right-hand side of \eqref{Steinpur} equals 
\begin{eqnarray}\label{equivStein}\phantom{AAAA}c_1(h)\sqrt{\Var \left[\sum_{i=1}^m X_i\sum_{j\in M_i}X_j\right] } + c_2(h)\EW\left[\sum_{i=1}^m |X_i|\left(\sum_{j\in M_i}X_j\right)^2 \right],
\end{eqnarray}
which in turn is equal to the right-hand side of \eqref{newY}. This concludes the proof of Proposition~\ref{propYnew}.
\end{proof}
\noindent
\begin{proof}[Proof of Theorem~\ref{asGausscolored}] It suffices to show that for all $d \in \mathbb N$ and $(\alpha_1, \ldots, \alpha_d) \neq (0,\ldots, 0)$,  the linear combination 
\begin{equation}\label{lincombS}
\bar Z_m := \alpha_0^{} Z^{(m)}_{\lfloor\rho_1^{} m\rfloor} + \alpha_1\left(Z^{(m)}_{\lfloor\rho_2^{} m\rfloor}-Z^{(m)}_{\lfloor\rho_1^{} m\rfloor }\right) + \cdots +  \alpha_d\left(Z^{(m)}_{m}-Z^{(m)}_{\lfloor \rho_{d-1}^{} m\rfloor }\right) 
\end{equation} is asymptotically Gaussian as $m \to \infty$. For  $m\in \N$ we put
\begin{equation}\label{defan}
a_i^{(m)} := \alpha_g \mbox{ \rm  if } \lfloor \rho_{g-1} m\rfloor  < i \le \lfloor \rho_{g} m\rfloor,    \quad   i=1,\ldots m;   \quad g=1,\ldots, d.
\end{equation}
It is then readily checked that $\bar Z_m$ defined in \eqref{lincombS} satisfies
\begin{equation}
\label{defSn2}\bar Z_m = \sum_{i=1}^m a_i^{(m)} Y_i^{(m)}, \quad m \in \N,
\end{equation}
and thus fits into the frame of Proposition \ref{propYnew}, with $\bar S_m := \bar Z_m$. To use this proposition we will show that under the assumptions \eqref{drei}, \eqref{vier}, \eqref{Covineqn} and \eqref{variancecomp} and with $a_i = a_i^{(m)}$ from \eqref{defan}, both summands in the right-hand side of \eqref{newY} converge to $0$ as $m \to \infty$. For notational convenience we will for the rest of this proof suppress the superscript~$m$ in the equivalence relation $\stackrel m\sim$, in the coefficients $a_i^{(m)}$  and in the random variables $Y_i^{(m)}$, $\mathscr I^{(m)}$,  $\mathscr J^{(m)}$,  $\mathscr K^{(m)}$,  $\mathscr L^{(m)}$.
  
  For a constant $C$ not depending on $m$ we have
\begin{equation}\label{Ysquare}
\Var \left[ \sum\limits_{i,j=1}^m Y_i^2a_ia_j I_{\{i\sim j\}}\right] \le C \Var \left[ \sum\limits_{i,j=1}^m Y_i^2 I_{\{i\sim j\}}\right],
\end{equation}
\begin{equation}\label{Ycube}
\EW\left[ \sum\limits_{i,j,k=1}^m |Y_i|^3 |a_i|a_ja_kI_{\{i\sim j\sim k\}} \right]  \le C\EW\left[|Y|^3\right]\sum\limits_{i,j,k=1}^m \WS(i\sim j \sim k).
\end{equation}
In order to bound $\Var \left[ \sum\limits_{i,j=1}^m Y_i^2 I_{\{i\sim j\}}\right]$ from above, we decompose the variance with respect to~$\mathscr P^{(m)}$ and first note that
\begin{equation*}
\EW\left[ \left. \sum\limits_{i,j=1}^m Y_i^2 I_{\{i\sim j\}}\right\vert \mathscr P^{(m)}\right] =  \EW\left[Y^2\right] \sum\limits_{i,j=1}^m  I_{\{i\sim j\}}.
\end{equation*}
The variance of the latter is 
\begin{equation*} \EW\left[Y^2\right]^2 \sum_{i,j,k,\ell \in [m]} \,  \Cov\left[I_{\{i\sim j\}}, I_{\{k\sim \ell\}}\right]
\end{equation*}
which by assumption  \eqref{Covineqn} is not larger than
\begin{equation} \label{EVar}\EW\left[Y^2\right]^2    \sum_{i,j,k,\ell \in [m]} \WS(i\sim j \sim k\sim \ell).
\end{equation}
Next we note that
\begin{eqnarray*}\label{VarE}\begin{split}
\Var\left[ \left. \sum\limits_{i,j=1}^m Y_i^2I_{\{i\sim j\}}\right\vert \mathscr P^{(n)}\right] &=&&  \sum\limits_{i,j,k,\ell=1}^m \Cov\left[\left.Y_i^2I_{\{i\sim j\}},Y_k^2 I_{\{k\sim \ell\}} \right\vert \mathscr P^{(m)}\right] \\
&=&& \Var\left[Y^2\right] \sum\limits_{i,j,k,\ell =1}^m  I_{\{i\sim j\sim k \sim \ell\}}.
\end{split}
\end{eqnarray*}
Taking expectation of the latter and adding this to \eqref{EVar} we obtain \begin{equation*}
\Var \left[ \sum\limits_{i,j=1}^m Y_i^2I_{\{i\sim j\}}\right] \le  \EW\left[Y^4\right]  \sum_{i,j,k,\ell \in [m]} \WS(i\sim j \sim k\sim \ell) = O\left(m^4 \WS\left(\mathscr I \sim \mathscr J \sim \mathscr K \sim \mathscr L\right)\right),\end{equation*}
which because of \eqref{vier} is $o\left(m^2  \WS\left(\mathscr I \sim \mathscr J\right)\right)^2$.

Likewise the right-hand side of \eqref{Ycube} is $o\left(m^2  \left(\WS\left(\mathscr I \sim \mathscr J\right)\right)^{3/2}\right)$. From \eqref{defSn} and \eqref{defSn2} we get 
\begin{equation}\label{sigmanstrich}
\bar \sigma_m^2 = \sum_{i,j\in [m]}a_i^{(m)}a_j^{(m)} \Cov\left[Y_i,Y_j\right] = \Var[Y] \sum_{i,j\in[m]}a_i^{(m)}a_j^{(m)} \WS(i\sim j), 
\end{equation}
which due to assumption \eqref{variancecomp} is bounded from below by $\widetilde C m^2\WS(\mathscr I \sim \mathscr J)$.
Thus under the assumptions of Theorem \ref{asGausscolored}  the right-hand side of \eqref{newY} converges to 0 as $m\to \infty$, which shows that Theorem \ref{asGausscolored} is a consequence of Proposition~\ref{propYnew}. 
\end{proof}

\section{Pair coalescence probabilities: Proof of Proposition  \ref{probrel}}\label{Proofprobrel}
   We now return to the setting of Section~\ref{secIntro}. Let $R^{(i)}_k$, $i\in \mathbb Z$, $k\in \mathbb N$ be independent copies of $R$, and define
\begin{equation}\label{defA}
A_i := \left\{n\in \mathbb Z : i-R_1^{(i)} - \cdots -R_j^{(i)} = n \mbox{ for some } j \ge 0\right\}.
\end{equation}
Note that the  $A_i$ are independent and thus can be seen as decoupled versions of the ancestral lineages of the individuals $i \in \mathbb Z$. In particular they do not coalesce if they meet. Decomposing with respect to the most recent collision time one obtains immediately (cf. \cite[p. 711]{HS})  that for $i > 0$
\begin{equation}\label{PE}
\WS\left(A_0\cap A_i\neq \emptyset\right)\sum_{m \ge 0} q_m^2 =   \EW\left[\left\vert A_0 \cap A_i\right\vert\right] = \sum_{m \ge 0} q_mq_{m+i}, 
\end{equation}
 hence 
 \begin{equation}\label{connection}
 \WS(0\sim i) = \frac{\sum_{m \ge 0} q_mq_{m+i}}{\sum_{m \ge 0} q_m^2}, \quad i \ge 0. 
 \end{equation}
We will now assume (in accordance with the assumptions in Proposition~\ref{probrel}) that the weights $q_n$ of the renewal measure defined in \eqref{defqn} have the property
\eqref{qas}. Under this condition we will prove
\begin{proposition}\label{PropLemma3c}
As $m\to \infty$,
\begin{equation} \label{claimprop}
\sum_{j\ge 1} q_j q_{m+j} \sim  \frac 1{\Gamma(\alpha)^2\Gamma(1-\alpha)^2}  \frac { m^{2\alpha-1}}{L(m)^2} \int_0^\infty (1+x)^{\alpha-1} x^{\alpha-1} \mathrm{d}x.
\end{equation}
\end{proposition}
The asymptotics \eqref{assim} claimed in Proposition~\ref{probrel} is  immediate from \eqref{connection} combined with \eqref {claimprop}.
\\\\
The remainder of this section is devoted to the proof of Proposition \ref{PropLemma3c}.
We will prove \eqref{claimprop} first under a special assumption on  the Karamata representation of the slowly varying function $L$.
\begin{lemma} \label{Karamatalemma}
Let $\alpha \in (0,\frac 12)$, and consider
\begin{equation} \label{eq1}
r_n := n^{\alpha-1} K(n)
\end{equation}
where $K(n)$ is of the form
\begin{equation}\label{eq2}
K(n) = \exp\left(\int_B^n \frac{l(t)}t\, \mathrm{d}t\right)
\end{equation}
with $B$ a positive constant and $l(t)$, $t\ge B$, a bounded measurable function converging to $0$ as $t\to \infty$.
Then $(r_n)$ is ultimately decreasing, and 
\begin{equation} \label{eq3}
\sum_{j\ge 1} r_j r_{i+j} \sim K(i)^2 i^{2\alpha-1} \int_0^\infty (1+x)^{\alpha-1} x^{\alpha-1} \mathrm{d}x \quad\mbox{as } i \to \infty,
\end{equation}
with
\begin{equation}\label{Betaintegral}
\int_0^\infty (1+x)^{\alpha-1} x^{\alpha-1} dx =  B(\alpha, 1-2\alpha) = \frac{\Gamma(\alpha) \Gamma(1-2\alpha)}{\Gamma(1-\alpha)}.
\end{equation}
\end{lemma}
\begin{proof}
a) The equality \eqref{Betaintegral} is readily checked by substituting $y= \frac x{1+x}$. 

b) The fact that $(r_n)$ is ultimately decreasing follows from \eqref{eq1} together with the Karamata representation \eqref{eq2} of $K(n)$. To see this, we argue as follows, putting $\beta := 1- \alpha$. Since $l(t)$ tends to zero for $t\rightarrow \infty$ we know that there exists $n_0\in\mathbb N$ such that for all $t\geq n_0$ one has $l(t)<\beta$. This implies
\begin{eqnarray*}
  \frac{r_n}{r_{n+1}} &=& \frac{n^{-\beta}}{(n+1)^{-\beta}} \cdot \frac{K(n)}{K(n+1)}\\
                      &=& \frac{ n^{-\beta}}{(n+1)^{-\beta }} \exp\left( - \int_n^{n+1} \frac{l(t)}{t} \mathrm dt \right)\\
                      &=& \exp\left( \beta \left( \ln (n+1)-\ln(n) \right) - \int_n^{n+1} \frac{l(t)}{t} \mathrm dt \right)\\
                      &=& \exp\left( \int_n^{n+1} \left( \frac{\beta}{t}-\frac{l(t)}{t}
                       \right) \mathrm dt \right).\\ 
\end{eqnarray*}
Since by assumption the integrand on the right-hand side is strictly positive for $n\geq n_0$, we obtain that $(r_n)_n$ is decreasing for $n\geq n_0$. 
\\\\
c) In view of \eqref{eq1}, the claimed asymptotics \eqref{eq3} is equivalent to
\begin{equation}\label{eq4}
\lim_{i\to \infty} \frac 1i \sum_{j\ge 1} \frac {r_j}{r_i} \frac{r_{i+j}}{r_i} = \int_0^\infty (1+x)^{\alpha-1} x^{\alpha-1} \mathrm{d}x.\end{equation}
We now set out to prove \eqref{eq4}. To this purpose we show first that for all $\varepsilon > 0$ there exists an $N \in \mathbb N$ such that for all sufficiently large $i$

\begin{equation}\label{eq5}
\frac 1i \sum_{j\ge Ni} \frac {r_j}{r_i} \frac{r_{i+j}}{r_i}< \varepsilon.
\end{equation}
Since $(r_n)$ is ultimately decreasing, \eqref{eq5} will follow if we can show that exists an $N \in \mathbb N$ such that for all sufficiently large $i$
\begin{equation}\label{eq6}
\frac 1i \sum_{j\ge Ni} \left(\frac {r_j}{r_i}\right)^2< \varepsilon.
\end{equation}
Again because of the ultimate monotonicity of $(r_n)$, the left-hand side. of \eqref{eq6} is for sufficiently large $i$ bounded from above by
\begin{equation}\label{eq7}
 \sum_{m=N}^\infty \left(\frac {r_{mi}}{r_i}\right)^2 = \sum_{m=N}^\infty m^{2\alpha - 2} \left(\frac{K(mi)}{K(i)}\right)^2.
\end{equation}
Using \eqref{eq2} one obtains that for any $\delta >0$ and $i$ so large that  $l(t)< \delta$ for all $t\geq i$,
\begin{equation}\label{eq8}
 \frac{K(mi)}{K(i)} = \exp\left( \int_i^{mi} \frac{l(t)}{t} \mathrm dt \right)\leq \exp\left( \delta \left( \ln (mi)-\ln (i) \right) \right)  \le m^\delta,
\end{equation}
which implies by dominated convergence that for sufficiently large $N$ the right-hand side of \eqref{eq7} is smaller than $\varepsilon$ for all sufficiently large $i$. We have thus proved~\eqref{eq5}.

Next we show that for all $\varepsilon > 0$ there exists an $\eta \in \mathbb N$ such that for all sufficiently large $i$
\begin{equation}\label{eq9}
\frac 1i \sum_{j\le \eta i} \frac {r_j}{r_i} \frac{r_{i+j}}{r_i}< \varepsilon.
\end{equation}
Again by ultimate monotonicity of $(r_n)$ which gives us $\frac{r_{i+j}}{r_i}\leq 1$ for $i$ large enough, for this it suffices to show that for all $\varepsilon > 0$ there exists an $\eta > 0 \in \mathbb N$ such that for all sufficiently large $i$
\begin{equation}\label{eq10}
\frac 1i \sum_{j\le \eta i} \frac {r_j}{r_i}< \varepsilon.
\end{equation}
From \cite[Theorem~5 on p.~447]{fellervol2} we obtain that 
\begin{equation}\label{eq11}
 \sum_{j\le \eta i} r_j \sim \frac 1\alpha (\eta i)^\alpha K(\lfloor \eta i\rfloor) \quad \mbox{as } i\to \infty,
\end{equation}
and hence
\begin{equation}\label{eq12}
 \frac 1i \sum_{j\le \eta i} \frac{r_j}{r_i} \sim \frac 1\alpha \eta ^\alpha \frac{K(\lfloor \eta i\rfloor)}{K(i)}\sim \frac{\eta^\alpha}{\alpha}  \quad \mbox{as } i\to \infty,
\end{equation}
which proves \eqref{eq10}, and hence also \eqref{eq9}. (The last asymptotic is by the fact that $K$ is slowly varying.)

In view of \eqref{eq1}, \eqref{eq5} and \eqref{eq9}, for proving \eqref{eq3} it remains to show that
\begin{equation}\label{eq13}
\lim_{i\to \infty}\frac 1i \sum_{\eta i \le j\le Ni} 
\frac {K(j)}{K(i)} \frac{K(i+j)}{K(i)} \left(1+\frac ji\right)^{\alpha-1}   \left(\frac ji\right)^{\alpha-1}   = \int_\eta^N (1+x)^{\alpha-1} x^{\alpha-1} \mathrm{d}x.
\end{equation}
From \eqref{eq2} one derives that
\begin{equation*}\label{eq14}
\lim_{i\to \infty} \sup_{\eta i \le j \le (N+1)i}  \left | \frac {K(j)}{K(i)} - 1\right | = 0.
\end{equation*}
Hence \eqref{eq13} boils down to a convergence of Riemann sums to its integral limit, and the proof of  Lemma \ref{Karamatalemma} is done.  
\end{proof}
Let us now complete the proof of Proposition \ref{PropLemma3c}.
\begin{proof}
The asymptotics \eqref{qas} can be rewritten as
\begin{equation*} \label{eq1_1}
q_n = C_n n^{\alpha-1} \tilde L(n)
\end{equation*}
where $\tilde L(n)$ is a slowly varying function and $C_n \to \frac 1{\Gamma(\alpha)\Gamma(1-\alpha)}  > 0$.

As in the proof of Lemma \ref{Karamatalemma} it suffices to show that
\begin{equation}\label{eq16}
\frac 1i \sum_{j\ge 1} \frac {q_j}{q_i} \frac{q_{i+j}}{q_i} \to \int_0^\infty (1+x)^{\alpha-1} x^{\alpha-1} \mathrm{d}x.
\end{equation}
Because of the Karamata representation theorem (see e.g. Theorem~1.3.1 in \cite{BGT}) there exists a $K(n)$ satisfying \eqref{eq2} and a sequence $D_n$ converging to a positive constant $D$ such that
\begin{equation}\label{eq17}
\tilde L(n) = D_n K(n), \quad n=1,2,\ldots
\end{equation}
Defining $r_n$ as in \eqref{eq1} we have
\begin{equation*}
q_n = D_n C_n r_n, \quad n=1,2,\ldots
\end{equation*}
Since the asymptotics of neither the left-hand side of \eqref{eq4} nor that of the left-hand side of \eqref{eq16} reacts to the omission of a fixed finite number of summands, we see that  \eqref{eq4} carries over to~\eqref{eq16}. 
\end{proof}

\section{Depth of most recent common ancestor: Proof of Proposition \ref{propMRCA}}\label{ProofMRCA}
In this section we will assume that the weights $q_n$ of the renewal measure defined in \eqref{defqn} obey the asymptotics
\eqref{qas}, see the discussion after equation~\eqref{qas}.
  For $i, m \in \mathbb N$ we set
  \begin{eqnarray*}
    f_i(m)=\WS\left(\mathscr M(0,i) =-m \right), \quad 
    F_i(m)=\WS\left(\mathscr M(0,i) \geq-m \right).
  \end{eqnarray*}
 For the independent  couplings $A_i$, $i\in \mathbb Z$, of the ancestral lines of $0$ and $i$ as defined in  \eqref{defA} we have for all $r > 0$ and $i\in \mathbb Z$
  \begin{eqnarray*}
    \sum_{k=0}^{ri} q_k q_{k+i} 
     &=&\EW\left[ | A_0 \cap A_i \cap \{0,\ldots, -ri \} | \right]\\
    &=& \sum_{k=0}^{ri} \left(f_i(k) \sum_{l=0}^{ri-k} q_l^2\right)
   \quad  \le \quad  \left( \sum_{l=0}^{ri} q_l^2  \right) F_i(ri),
  \end{eqnarray*}
and consequently
  \begin{eqnarray}\label{tailest}
   \WS(D_i \le ri) =  F_i(ri) \geq \frac{ \sum_{k=0}^{ri} q_k q_{k+i} }{ \sum_{l=0}^{ri} q_l^2  }.
  \end{eqnarray}
As in the proof of Proposition \ref{PropLemma3c} we obtain
  \begin{equation}\label{incomp}
    \sum_{k=0}^{ri} q_k q_{k+i} \sim \frac{i^{2\alpha -1}}{L(i)^2} \frac{1}{\Gamma(1-\alpha)^2\Gamma(\alpha)^2} \int_0^r x^{\alpha -1}(1+x)^{\alpha -1}\mathrm{d}x \quad  \mbox{ as } i \to \infty.
  \end{equation}
Together with \eqref{tailest} and the asymptotics \eqref{assim} this gives
  \begin{equation}\label{lowerest}
    \liminf_{i\to \infty}\WS\left( D_i\le ri\, \vert\,  0\sim i\right) \geq   B(\alpha, 1-2\alpha)^{-1} \int_0^r x^{\alpha -1}(1+x)^{\alpha -1}\mathrm{d}x.
      \end{equation}
For the upper estimate we choose some arbitrary $\varepsilon>0$ and observe
  \begin{eqnarray*}
   \sum_{k=0}^{(r+\varepsilon)i} q_k q_{k+i}
    &=& \EW\left[ | A_0 \cap A_i \cap \{0,\ldots, -(r+\varepsilon)i \} | \right]\\
    &=& \sum_{k=0}^{(r+\varepsilon)i} \left(f_i(k) \sum_{l=0}^{(r+\varepsilon)i-k} q_l^2\right)
    \quad \geq \quad \sum_{k=0}^{ri} \left(f_i(k) \sum_{l=0}^{(r+\varepsilon)i-k} q_l^2\right)\\
    &\geq& \left( \sum_{l=0}^{\varepsilon i} q_l^2  \right) \sum_{k=0}^{ri} f_i(k)
    \qquad =  \quad  \left( \sum_{l=0}^{\varepsilon i} q_l^2  \right) F_i(ri).
  \end{eqnarray*}
  Using \eqref{incomp}, now with  $r+\varepsilon$ instead of $r$, we get 
 \begin{equation*}\label{upperest}
    \limsup_{i\to \infty}\WS\left( D_i\le ri\, \vert\,  0\sim i\right)  \le   B(\alpha, 1-2\alpha)^{-1} \int_0^{r+\varepsilon} x^{\alpha -1}(1+x)^{\alpha -1}\mathrm{d}x.
      \end{equation*}
Since $\varepsilon$ was arbitrary, this together with  \eqref{lowerest}  gives the assertion of Proposition~\ref{propMRCA}.
\section{Triplet and quartet coalescence probabilities: Proof of Proposition \ref{3and4}}\label{Proofthreefour}
In this section we will assume that the weights $q_n$ of the renewal measure defined in \eqref{defqn} obey the asymptotics
\eqref{qas}.
  We now turn to the asymptotic analysis of triplet and quartet coalescence probabilities.
  \subsection{Triplet coalescence probabilities}\label{tripletproof}
\begin{lemma}\label{prop:3er}
  For all $r>0$ and $i\in\N$ we have for a slowly varying function $\tilde L$ not depending on $r$
  \begin{equation}\label{tripletest}
   \WS\left(0\sim i \sim \lfloor(1+r)i\rfloor  \right)\leq   \left(r^{\alpha-1}+r^{2\alpha-1} \right) \tilde L(i)\,  i^{4\alpha -2}.
\end{equation}
\end{lemma}
\begin{proof}[Proof]
Let $(A_k)$ be the independent (non-merging) ancestral lineages defined in \eqref{defA}.
  We set
  \begin{equation*}
  A_{k,\ell} :=\begin{cases} A_{\max(A_k\cap A_\ell)} &\mbox{ if } A_k \cap A_\ell \neq \emptyset,
  \\ \emptyset  &\mbox{ if } A_k \cap A_\ell = \emptyset, \end{cases}
  \end{equation*}
  \begin{equation*}\label{defBkl}
  B_{k, \ell}:= \bigcup_{ g \in A_k \cap A_\ell } A_g.
\end{equation*}
In words, $B_{k, \ell}$ is the union of the (non-merging) ancestral lineages starting at all the points of intersection of $A_k$ and $A_\ell$. Distinguishing the 3 shapes of the ancestral tree of the individuals $0$, $i$ and $(1+r) i$ on  the event $E:= \{0\sim i \sim (1+r)i \}$ (and omitting  the Gauss brackets in $\lfloor (1+r)i\rfloor$ etc. for the sake of readability)  we have by
subadditivity 
\begin{eqnarray} 
  && \WS\left( 0\sim i \sim (1+r) i \right)\nonumber\\
  &\leq& \WS\left(E; \mathscr M(0,i)\ge  \mathscr M(0, (1+r)i   \right)+ \WS\left(E; \mathscr M(0,(1+r)i)\ge \mathscr M(0,i)   \right)\nonumber
  \\ && 
  + \WS\left(E;\mathscr M(i,(1+r)i)\ge \mathscr M(0,i)   \right)
  \nonumber
  \\
  &\leq&  \WS\left(A_{0,i} \cap A_{(1+r)i}\neq \emptyset  \right)+ \WS\left( A_{0,(1+r)i} \cap A_{i}\neq \emptyset \right)\nonumber + \WS\left(  A_{i,(1+r)i} \cap A_{0}\neq \emptyset \right) 
  \\
  &\leq&  \EW\left[ | A_{0,i} \cap A_{(1+r)i}|  \right]+ \EW\left[ | A_{0,(1+r)i} \cap A_{i}|  \right] + \EW\left[ | A_{i,(1+r)i} \cap A_{0}|  \right] \label{eq:triple_2}
  \\ 
  &\leq& \EW\left[ | 
         B_{0,i} \cap A_{(1+r)i}|  \right]+ \EW\left[ | B_{0,(1+r)i} \cap A_{i}|  \right] + \EW\left[ | B_{i,(1+r)i} \cap A_{0}|  \right] .\label{eq:triple}\nonumber
         \\&=&\sum_{m\geq 0}q_m q_{m+i} \sum_{n\geq 0 }q_n q_{n+(1+r)i+m}+\sum_{m\geq 0}q_m q_{m+(1+r)i} \sum_{n\geq 0 }q_n q_{n+i+m} \nonumber
         \\&&+
\sum_{m\ge 0}q_m q_{ri+m}\sum_{n\geq 0} q_n q_{n+i-m}. \nonumber
\end{eqnarray}
An inspection of the third summand on the  right-hand side leads to
     \begin{eqnarray*}
&&\sum_{m\ge 0}q_m q_{ri+m}\sum_{n\geq 0} q_n q_{n+i-m} \\
  & =& q_m q_{ri+i}\sum_{n\geq 0} q_n q_n  + \sum_{m\neq i}q_m q_{ri+m}\sum_{n\geq 0} q_n q_{n+i-m}\\
  & \le& C \left(q_i q_{(1+r)i} + \sum_{m\neq i}\frac{m^{\alpha-1}}{L(m)}\frac{(m+ri)^{\alpha-1}}{L(m+ri)}\frac{|i-m|^{2\alpha-1}}{L(|i-m|)^2} \right)
       \end{eqnarray*}
       \begin{eqnarray*}
  & \sim& C \left(\frac{i^{2\alpha-2}}{L(i)^2} + \frac{i^{4\alpha-2}}{L(i)^4}  \int_0^\infty x^{\alpha-1}(x+r)^{\alpha-1} |1-x|^{2\alpha-1} \mathrm{d}x\right)  \\
  &\le&   \tilde L(i) i^{4\alpha-2} r^{\alpha -1},
     \end{eqnarray*}
     where $\tilde L$ is a slowly varying function that dominates $C\left(\tfrac 1{L^2} + \tfrac 1{L^4}\right)$, and where the asymptotics is justified in the same way as \eqref{eq13} was derived first for slowly varying functions satisfying \eqref{eq2} and then, using the Karamata representation, for general slowly varying functions obeying \eqref{eq17}. The first and the second summand on the r.h.s of \eqref{eq:triple_2} can be analysed in an analogous manner, leading to the bounds $\tilde L(i) i^{4\alpha-2} r^{2\alpha -1}$  and $\tilde L(i) i^{4\alpha-2} r^{\alpha -1}$, respectively. 
\end{proof}
\begin{proof}[Proof of Proposition \ref{3and4} Part 1]
We  set out to show \eqref{trip}, and first observe that
\begin{equation}\label{eq63}
\sum_{i=0}^n \sum_{j=0}^n \WS\left(  0 \sim i  \sim j \right)\\
      = \sum_{i=0}^n  \WS\left(  0 \sim i  \right) + 2 \sum_{i=1}^n \sum_{j=1}^{n-i} \WS\left( 0 \sim i  \sim  i+j \right).
\end{equation}
By Lemma~\ref{prop:3er}
we get, noting that $r^{\alpha-1}+r^{2\alpha-1} \le 2r^{\alpha-1}+1 $,
  \begin{eqnarray} \label{secondsum}
  \begin{split}
\sum_{i=1}^n \sum_{j=1}^{n-i} \WS\left( 0 \sim i  \sim  i+j \right) &\le
  \sum_{i=1}^n \sum_{j=1}^{n-i} \tilde L(i)i^{4\alpha -2} \left( \left( \frac{j}{i} \right)^{\alpha -1}+ 1  \right)\\
      &\leq \tilde L(n)\left(\sum_{i=1}^n i^{3\alpha -1 } (n-i)^{\alpha  }+  n\sum_{i=1}^n i^{4\alpha -2 }\right)= O\left(n^{4\alpha+\delta}\right)
      \end{split}  
      \end{eqnarray}
 for each $\delta >0$. From \eqref{eq63} combined with \eqref{assim} and \eqref{secondsum} we obtain for each $\delta >0$  the estimate
  \begin{equation}
\label{analogouseq}\sum_{i=0}^n \sum_{j=0}^n \WS\left(  0 \sim i  \sim j \right) \le C n^{4\alpha+\delta},
\end{equation}
where the constant $C$ depends on $\delta$ but not on $n$. An analogous calculation for $k$ in place of $0$ shows that also
\begin{equation}\label{uniformink}
\sum_{i=0}^n \sum_{j=0}^n \WS\left(  k \sim i  \sim j \right) \le C n^{4\alpha+\delta}, \quad k \in [n],
\end{equation}
where $C$ can be chosen uniformly in $n$ and in $k  \in [n]$. (An intuitive reason for this uniformity comes from the fact that for each $k \in [n]$ and small $\varepsilon > 0$, the big majority of the pairs $(i,j) \in [n]^2$  leads to pairwise distances $|i-j|$, $|i-k|$ and $|j-k|$ that are all between $\varepsilon n$ and $n$.)
With this uniformity in $k$, \eqref{trip}  follows directly from \eqref{uniformink}.
\end{proof}

 \subsection{Quartet coalescence probabilities}
 \begin{lemma}\label{prop:4er}
  For all  $r_2>r_1>0$ and $i\in\N$ we have for a slowly varying function $\bar L$ not depending on $r$
\begin{eqnarray}
  &&\WS\left( 0 \sim i\sim \lfloor (1+r_1)i\rfloor \sim \lfloor(1+r_2)i\rfloor  \right)\nonumber \\
  &\leq&   \left(r_1^{\alpha-1}+ r_1^{2\alpha-1}+r_2^{\alpha-1}+ r_2^{2\alpha-1}  \right) \bar L(i) i^{6\alpha -3}.\label{quartetest}
   \end{eqnarray}
\end{lemma}
\begin{proof}[Proof]
  Again let $(A_i)$ be the independent (non-merging) ancestral lineages defined in \eqref{defA}.
Let
  $B_{k, \ell}$ be as in
\eqref{defBkl}
and set
  \begin{equation*}
  B_{[k,\ell]j}:=  \bigcup_{ m \in B_{k,\ell} } \bigcup_{ g \in A_m \cap A_j } A_g.
\end{equation*}
We fix $i\in\N$ and set  with $j_1:= (1+r_1)i$, \, $j_2 := (1+r_2)i$, 
\begin{eqnarray*}Q:= &&\{(0,i,j_1, j_2),(0,i,j_2,j_1), (0,j_1,i,j_2),  (0,j_1,j_2,i),\\&& (0,j_2,i,j_1),(0,j_2,j_1,i), (i,j_1,0,j_2), (i,j_1,j_2,0),\\&& (i,j_2,0,j_1), (i,j_2,j_1,0), (j_1,j_2,0,i),(j_1,j_2,i,0)\}
\end{eqnarray*}
We now argue in a similar way as in the proof of Proposition \ref{prop:3er}. By subadditivity we get
\begin{eqnarray*}
 \WS\left( 0 \sim i\sim (1+r_1)i \sim (1+r_2) \right)
  \leq \sum_{(k_1,k_2,k_3,k_4) \in Q} \EW\left[ \vert B_{[k_1,k_2]k_3} \cup A_{k_4} \vert \right].      \end{eqnarray*}
By the very same arguments as in the proof of Lemma \ref{prop:3er} one checks that each of the twelve summands is bounded by the right-hand side of \eqref{quartetest}.
\end{proof}
\begin{proof}[Proof of Proposition \ref{3and4} Part 2]
We  are now going to prove \eqref{quadr}, and first set out to show that for all $\delta > 0$
   \begin{equation}\label{sumprob}
     \sum_{i=1}^n \sum_{j=1}^{n-i} \sum_{\ell=1}^{n-i-j} \WS\left( 0 \sim i \sim i+j \sim i+j+\ell \right)= O\left(n^{6\alpha+\delta}\right).
   \end{equation}
   Setting $r_1=\tfrac{j}{i}, r_2 = \tfrac{j+l}{i}$ we have $r_1^{\alpha-1}+ r_1^{2\alpha-1}+r_2^{\alpha-1}+ r_2^{2\alpha-1} \le 4(r_1^{\alpha-1}+1)$ and by Lemma~\ref{prop:4er} the left-hand side of \eqref{sumprob}  is bounded from above by
   \begin{equation*}\label{boundsum}
    C  \sum_{i=1}^n \bar L(i) i^{6\alpha -3} \sum_{j=1}^{n-i} \sum_{\ell=1}^{n-i-j} 
    \left[ 
    \left( \frac{j}{i}\right)^{\alpha-1} +  1
    \right].
   \end{equation*}
   By  arguments analogous to those leading to \eqref{eq63} in the proof of Part 1, this implies \eqref{sumprob}.
    As in the proof of Part 1  we can argue that, as $n\to \infty$   the terms
\begin{equation*}
    \sum_{i=0}^n \sum_{j=0}^n \sum_{\ell=0}^n \WS\left( k \sim i \sim j \sim \ell  \right)
  \end{equation*}
  are of the same order uniformly in for all $k\in \{0,1,\ldots, n\}$, so it is enough to look at the case $k=0$. Also, from \eqref{assim} and \eqref{trip} it is clear that we may  restrict to pairwise distinct $i, j, \ell$. Thus, similar as in Part 1,  \eqref{quadr} follows from \eqref{sumprob}.    \end{proof}
 \section{A covariance estimate: Proof of Lemma \ref{Covest}}\label{ProofCovest}
  For $i=j$ or $k=\ell$ the assertion of Lemma \ref{Covest} is clearly true because then the left-hand side. of \eqref{Covineq} vanishes. For $i=k$ we have
  \[ \Cov[I_{\{i\sim j\}}, I_{\{i\sim \ell\}}] = \WS(i\sim j\sim \ell)-\WS(i\sim j) \WS(i\sim \ell) \le \WS(i\sim j \sim \ell).\]
  Hence we may assume without loss of generality that $i,j,k,\ell$ are pairwise distinct. 
We then have
  \begin{equation}
    \Cov\left[ I_{ i\sim k },  I_{ j\sim \ell } \right]= \WS\left(   i \sim k \cap j \sim \ell \right)-  \WS\left(   i \sim k  \right) \WS\left(  j \sim \ell \right)\label{eq:covquado}
  \end{equation}
  and
  \begin{equation}\label{eq:decomquadro}
    \WS\left(   i \sim k \cap j \sim \ell \right) = \WS\left(   i \sim k \not \sim  j \sim \ell \right)  + \WS\left(   i \sim k \sim  j \sim \ell \right). 
  \end{equation}
 By \eqref{eq:covquado} and \eqref{eq:decomquadro}, the inequality \eqref{Covineq} is immediate from the following  \begin{lemma}
    For pairwise distinct $i,j,k,\ell$ we have
    \begin{equation*}\label{beh1}
    \WS\left( i  \sim k\not \sim j \sim \ell \right) \le \WS\left( i \sim k  \right) \WS\left( j \sim \ell  \right)
  \end{equation*}
  \end{lemma}
  \begin{proof}[Proof]
 Let $A_g$, $g\in \mathbb Z$, be defined as in \eqref{defA}.   For $m \in \mathbb N$ and  \mbox{$i_1 < \ldots <  i_m \in \mathbb Z$} we define an $(A_{i_1}, \ldots, A_{i_m})$-measurable random graph $G^{\{i_1,\ldots,i_m\}}$ which is equal in distribution to the subgraph of $\mathcal G_\mu$ that is formed by the (possibly coalescing) ancestral lineages of $i_1, i_2, \ldots, i_m$. The construction of $G^{\{i_1,\ldots,i_m\}}$ is done inductively in the following ``lookdown'' manner: the ancestral lineage of $i_1$ is taken as $A_{i_1}$, correspondingly, we put   $ G^{\{i_1\}}:= A_{i_1}$. The ancestral lineage of $i_{h+1}$ is given by  $A_{i_{h+1}}$ as long as the latter did not  meet $G^{\{i_1,\ldots, i_h\}}$. At the time of the first (seen in backward time direction) collision of $A_{i_{h+1}}$ with  $G^{\{i_1,\ldots, i_h\}}$, the ancestral lineage of $i_{h+1}$ is continued by the lineage in $G^{\{i_1,\ldots, i_h\}}$ that starts in the meeting point (and the continuation of $A_{i_{h+1}}$ from there on is erased). An inspection of  $G^{\{i,j,k,\ell\}}$ reveals that
   \begin{eqnarray*}
&&\WS\left( \left\{ i \sim k \not \sim j \sim \ell  \right\} \right) 
\\
&=&  \WS\left(\left\{ A_i \cap A_k \not=\emptyset \right\}\cap \left\{ A_j \cap A_\ell \not=\emptyset \right\}\cap \left\{ G^{\{i, k\}} \cap  G^{\{j, \ell\}} = \emptyset \right\}\right)\\
&\le&  \WS\left(\left\{ A_i \cap A_k \not=\emptyset \right\}\cap \left\{ A_j \cap A_\ell \not=\emptyset \right\}\right),
   \end{eqnarray*}
     which because of mutual independence of the $A_g$ gives the assertion of the lemma.
   \end{proof}
   \section{Asymptotic Gaussianity in the Hammond-Sheffield urn: Proof of Theorem \ref{mainth}\ref{theoremmain:B}}\label{secmainproof}
  We are going to apply Theorem \ref{asGausscolored}, with $\mathscr P^{(m)}$ being  the partition on $[m]$ that is generated by the Hammond-Sheffield urn, i.e. by the equivalence class $\sim$ defined in \eqref{simdef}. For $d \in \N$ and $t_1 < \cdots < t_d$ as prescribed in Theorem \ref{mainth} we apply Theorem \ref{asGausscolored} with $m=m(n)=\lfloor t_d n\rfloor$ and $\rho_g := \lfloor t_g n\rfloor/\lfloor t_{d} n\rfloor$, $g=1,\ldots, d$.
Under condition \ref{theoremmain:B} of Theorem \ref{mainth}, Corollary~\ref{cor3and4} ensures the validity of assumptions  \eqref{drei} and \eqref{vier}, and Lemma~\ref{Covest} guarantees that \eqref{Covineqn} is fulfilled. It remains to check the assumption \eqref{variancecomp}.
Indeed, with 
\begin{equation*} \label{defax}
a(x) := \alpha_g \mbox{ if } \rho_{g-1}  < x \le \rho_{g},    \quad   x \in (0,1],   \quad g=1,\ldots, d,
\end{equation*}
because of  $\WS(i\sim j)= \WS(0\sim |i-j|) $ and in view of Proposition~\ref{probrel}   the left-hand side. of \eqref{variancecomp} has the asymptotics 
\begin{equation} \label{Riemannagain}
m^{-2}\sum_{i,j\in [m]}a^{(m)}_ia^{(m)}_j \WS(i\sim j) \sim \frac{C m^{2\alpha-1}}{ L(m)^{2}}   \int_0^1 \int_0^1 a(x)a(y) |x-y|^{2\alpha-1}\mathrm{d}x\, \mathrm{d}y \quad \mbox{as } m\to \infty
\end{equation}
The Riesz kernel $|x-y|^{2\alpha-1}$ is positive definite (see e.g. \cite{rieszkernel}), hence the integral term in \eqref{Riemannagain} is strictly positive, and consequently the left-hand side of \eqref{Riemannagain} is of the order $m^{2\alpha+1}$ as $m\to \infty$. Because of ~\eqref{asvar}, this is also the order of the right-hand side of  \eqref{variancecomp}.
\section{Tightness: Proof of Proposition \ref{proptight}}\label{tight}
Inspired by the proof of Theorem 1 in \cite{MR1849425}, which shows tightness of a different approximation scheme for fractional Brownian motion, we will make use of the following
\begin{lemma}[{\cite[Theorem~13.5]{measures2}}]\label{satz:billingsley}  Let $T>0$ and 
 $(\zeta^{(n)}(t))_{0\le t\le T}$  be continuous processes that converge to a continuous process $(\zeta(t))_{0\le t\le T}$  in the sense of finite dimensional distributions. Assume  that for  a nondecreasing continuous function $F$ on $[0,T]$, for some $\gamma> 1$, for all $0\leq s\leq t \leq u\le T$, for some  $N\in\N$ and all $n \geq N$
  \begin{equation}
    \label{eq:82}
    \EW\left[ \left| \zeta^{(n)}(t) -\zeta^{(n)}(s) \right| \cdot \left| \zeta^{(n)}(u)- \zeta^{(n)}(t)\right|   \right] \leq \left[  F(u)-F(s) \right]^{ \gamma}.
  \end{equation}
 Then $(\zeta^n(t))_{0\le t\le T}$ converges in distribution in $\left(C\left([0,T]\right), || \cdot ||_{L^\infty([0,T])}\right)$ to $(\zeta(t))_{0\le t\le T}$. 
\end{lemma}
\begin{proof}[Proof of Proposition~\ref{proptight}]
We first note the following immediate consequence of \eqref{sigman} and \eqref{asvar}:
  There exist constants $c, c'>0$ such that
  \begin{equation}
    \label{eq:69}
    c n^{2\alpha+1} L(n)^{-2} \le \sigma_n^2 \le c' n^{2\alpha+1} L(n)^{-2}, \quad \forall n\in\N. 
  \end{equation}
Next we fix  $0\leq s \leq t \leq u \leq T$ and $j,k,l\in\N$ satisfying
  \begin{equation*}
    \label{eq:48}
    \frac{j}{n} \leq s < \frac{j+1}{n},\quad \frac{k}{n} \leq t < \frac{k+1}{n}, \quad \frac{l}{n} \leq u < \frac{l+1}{n}.
  \end{equation*}
The definition of $S^{(n)}$ as a linear interpolation gives
  \begin{equation*}
    \label{eq:55}
  \sigma_n(S^{(n)}(t)- S^{(n)}(s))  = \sum_{m=j+1}^k Y_m + \left[ 1- \left( sn -j \right) \right]Y_{j+1} + \left[ tn -k \right] Y_{k+1}.
  \end{equation*}
  We note that
  $0\leq  \left[ 1- \left( sn -j \right) \right]\leq 1$, $0\leq \left[ tn -k \right]  \leq 1$ and
get that for some constants $c_1, c_2, c_3, c_4$
  \begin{eqnarray}
    \label{eq:72}
    \Var\left[ S^{(n)}(t) - {S}^{(n)}(s) \right]&\leq& \left[ 2\Var\left[  S^{(k-j)} \right] +2 c_{1} \Var\left[ Y_{j+1} + Y_{k+1} \right] \right] \sigma_n^{-2}\nonumber\\
                                                    &\leq& \left[ 2 c' (k-j)^{2\alpha+1} L(k-j)^{-2} + c_{2} \right] \ \sigma_n^{-2} \nonumber\\
                                                    &\leq& \left[ c_3(k-j)^{2\alpha+1} L(k-j)^{-2} \right]  \sigma_n^{-2}  \nonumber\\
    &\leq& c_4 \left( \frac{k-j}{n} \right)^{2\alpha+1} \left( \frac{ L(n) }{L(k-j)} \right)^2. \label{eq:72_3}
  \end{eqnarray}
  Here the first inequality holds because for any two square-integrable random variables $G_1$, $G_2$ one has $\Var\left[ G_1+G_2 \right] \leq 2\Var[G_1]+2 \Var[G_2]$, the    second and the last inequality hold because of \eqref{eq:69}, and the third one holds for some $c_3>0$, some  $N_1\in\N$ and all $n \ge N_1$ because $L$ is a slowly varying function. (Remember that $k,j,l$ depend on $n$ for fixed time points $s,t,u$.)  Analogously,
  \begin{equation}
    \label{eq:75}
    \Var\left[ S^{(n)}(u)- S^{(n)}(t) \right]\leq
    c_3\left( \frac{l-k}{n} \right)^{2\alpha+1} \left( \frac{ L(n) }{L(l-k)} \right)^2 .
  \end{equation}
  To use Lemma~\ref{satz:billingsley} we have to bound the expectation
  \begin{equation*}
    \label{eq:73}
    \EW\left[ \left| S^{(n)}(t) -  S^{(n)}(s)  \right| \cdot \left|  S^{(n)}(u) -  S^{(n)}(t)  \right| \right].
  \end{equation*}
  For $l\leq j+1$ we get that $|u-t|\leq \frac2n$ and $|t-s|\leq \frac2n$, so that by basic calculus $(u-t)(t-s)\leq (u-s)^{1+\alpha}$. By the definition of $\mu$ and the fact that $L$ is slowly varying one  has $L(n)\leq n^{-\alpha}$ for $n$ large enough. So linear interpolation gives:
  \begin{eqnarray*}
    \label{eq:13}
    &&\EW\left[ |S^{(n)}(t)- S^{(n)}(s)| \cdot | S^{(n)}(u)-S^{(n)}(s)| \right]\\
    &\leq& \left( \tilde{c} n^{-\frac12-\alpha}L(n) \right)^2  \left( (t-s) \cdot (u-s) \right)\\
    &\leq& \left( \tilde{c} n^{-\frac12} \right)^2  (u-s)^{1+2\alpha} 
    \, =\, \tilde{c}^2 \frac1n  (u-s)^{1+2\alpha}\\
    &\leq& \tilde{c}^2  (u-s)^{1+2\alpha}\, = \, \left[ c_0 u - c_0 s \right]^{1+2\alpha}
  \end{eqnarray*}
  for some $c_0>0$.
   Now assume  $l>j+1$. Cauchy-Schwarz and the estimates \eqref{eq:72_3}, \eqref{eq:75} yield
   \begin{eqnarray}
     \label{eq:74}
     &&\EW\left[  \left|  S^{(n)}(t) -  S^{(n)}(s)   \right| \cdot \left|  S^{(n)}(u) -  S^{(n)}(t)  \right|  \right]\nonumber\\
     &\leq& \sqrt{ \Var \left[  S^{(n)}(t) -  S^{(n)}(s) \right] }\sqrt{ \Var \left[  S^{(n)}(u) -  S^{(n)}(t) \right] }\nonumber\\
     &\leq&\sqrt{c_3 \left( \frac{k-j}{n} \right)^{2\alpha+1} \left( \frac{ L(n) }{L(k-j)} \right)^2 } \sqrt{ c_3 \left( \frac{l-k}{n} \right)^{2\alpha+1} \left( \frac{ L(n) }{L(l-k)} \right)^2}\nonumber\\
     &\leq& c_3\left( \frac{l-j}{n} \right)^{2\alpha+1} \left( \frac{L(n)}{L(k-j)L(l-k)} \right)^2 \nonumber\\
     &\leq& 2c_3\left( \frac{l-j}{n} \right)^{2\alpha+1} \le\quad   \left(2c_3 \cdot 2(u-s) \right)^{2\alpha+1} . \nonumber
   \end{eqnarray}
  (The third inequality holds because of $4|k-j| \cdot |l-k| \leq |l-j|^2$, and the fourth one holds for $n\geq N_2$ for some  $N_2 \in\N$ because $L$ is a slowly varying function.) 
  
  Thus,   $\zeta^{(n)}:= S^{(n)}$ fulfills condition \eqref{eq:82} with with $N:= \max\{ N_1, N_2 \}$, $\gamma := \alpha + 1/2$ and $F(x):= 4 \max\{c_0, c_3\} x$. In view of Corollary~\ref{findim}  we can thus apply Lemma~\ref{satz:billingsley} and conclude the assertion of Proposition \ref{proptight}.
 \end{proof}
 \begin{remark}\label{HSLemma} In \cite{HS} the functional convergence of $S^{(n)}$ was deduced from a  tail estimate (uniform in $n$) on $\max_{t\in [0,1]} S^{(n)}_t$, stated in \cite[Lemma 4.1]{HS}. The proof of this lemma given there relies on the statement that for a certain sequence $(r_n)$ the inequalities (4.11) in \cite{HS} imply boundedness of $(r_n)$. There are, however, examples of unbounded sequences which fulfill these inequalities. Still, things clear up nicely because the assertion of \cite[Lemma 4.1]{HS} is a quick consequence of  Corollary \ref{funcconv} and the Borell-TIS inequality.
 \end{remark}

 \section{Coalescence probabilities in long-range seedbanks}\label{secBGKS}
 In this section we will assume that the weights $q_n$ of the renewal measure defined in \eqref{defqn} obey the asymptotics
\eqref{qas}, see the discussion after equation~\eqref{qas}.
Following \cite{BGKS} we extend the model described in Section \ref{secIntro} as follows. For fixed $N \in \mathbb N$, the set of vertices of  $\mathcal G_\mu$ is now $\mathbb Z \times [N]$. The set of those vertices whose first component is $i$ constitutes the population of individuals living at time $i$. The parent of the individual $(i,k)$ is $(i-R_{i,k}, H_{i,k})$, where the random variables $R_{i,k}$ are independent copies of $R$ and the random variables $H_{i,k}$ are i.i.d.~picks from $[N]$. In words, each individual chooses its parent uniformly from a previous time with delay (or dormancy) distribution~$\mu$.  The corresponding urn model, which goes back to Kaj, Krone and Lascoux (\cite{kkj}), thus specialises to the Hammond-Sheffield urn for $N=1$.

Again we write $(i,k) \sim (j,\ell)$ if the two individuals $(i,k)$ and $(j,\ell)$ belong to the same connected component of $\mathcal G_\mu$.
Thanks to Proposition \ref{PropLemma3c}, which also provides a proof of \cite[Lemma 3.1 (c)]{BGKS}, we arrive at the following analogue of Proposition \ref{probrel} (see also \cite[Theorem~3(c))]{BGKS}:
\begin{proposition}
For all $k, \ell \in [N]$
\begin{equation} \label{LemmaBGKS}
    \WS\left((0,k)\sim (i,\ell)\right)\sim C_{\alpha,N} \frac {i^{2\alpha -1}}{L(i)^2} \quad \mbox{ as  } i\to \infty,  
     \end{equation} 
     where now
      \begin{equation*}\label{CalphaN}
     C_{\alpha,N} := \frac 1{N+\sum_{m\ge 1} q_m^2} \, \frac{\Gamma(1-2\alpha)}{\Gamma(\alpha)\Gamma(1-\alpha)^3}.
     \end{equation*}
\end{proposition} 
\begin{proof}
Denote by 
 $\tilde A_{0,k}$ and $\tilde A_{i, \ell}$  the decoupled ancestral lineages of the individuals $(0,k)$ and $(i,\ell)$ constructed in analogy to \eqref{defA}. Like in \eqref{PE} we observe
  \begin{equation*}
    \EW\left[ |\tilde A_{0,k} \cap \tilde A_{i,\ell}| \right]= \sum_{m\geq 0} \frac1N q_m q_{m+i}.\label{eq:eqtreff1_1}
  \end{equation*}
Decomposing at the most recent collision time of  $\tilde A_{0,k}$ and $\tilde A_{i, \ell}$ we get
  \begin{eqnarray*}
    \EW\left[ \tilde A_{0,k} \cap \tilde A_{0,\ell} \right] = \WS\left(  \tilde A_{0,k} \cap  \tilde A_{i,\ell}\not=\emptyset \right)\cdot \left( 1+ \sum_{m\geq 1}\frac1N q_m^2 \right) . \label{eq:eqtreff2_1}
  \end{eqnarray*}
The last two equalities combine to
  \begin{eqnarray*}
    &&\WS\left( A_{0,k} \cap A_{i,\ell}  \not=\emptyset \right)
        =\frac{ \ \sum_{m\geq 0} q_m q_{m+i}   }{ N+ \sum_{m\geq 1} q_m^2 }.
  \end{eqnarray*}
  The claimed asymptotics \eqref{LemmaBGKS} is now immediate from Proposition~\ref{PropLemma3c}. 
\end{proof}
\begin{remark}
For given natural numbers $n$ and $N$ let 
   $\mathscr I, \mathscr J, \mathscr K$ and \mbox{$\mathscr L$}
    be independent and  uniformly distributed on~$[n]\times [N]$.  In complete analogy to Proposition~\ref{3and4} one can derive that  for all $\delta>0$ and for a constant $C$ not depending on $n$ and $N$
 \begin{eqnarray*}
\WS\left(\mathscr I \sim \mathscr J \sim  \mathscr K\right) \qquad &\le& \frac C{N^2} n^{4\alpha-2+\delta}, \\
\WS\left(\mathscr I \sim \mathscr J \sim  \mathscr K  \sim  \mathscr L\right) &\le& \frac C{N^3} n^{6\alpha-3+\delta}.
 \end{eqnarray*}
 Consequently, along the lines of the proof of Theorem \ref{mainth} one obtains the convergence of an analogue of $\big(S^{(n)}\big)$ towards fractional Brownian motion also in the long-range seedbank model of  \cite{BGKS}.
\end{remark}
\paragraph{\bf Acknowledgment.}
We thank Matthias Birkner, Florin Boenkost, Adrian Gonz\'alez Casanova, Alan Hammond and Nicola Kistler for stimulating discussions and valuable hints. We also thank the \emph{Allianz f\"ur Hochleistungsrechnen Rheinland-Pfalz} for granting us access to the High Performance Computing \textsc{Elwetritsch}, on which simulations have been performed which inspired results in this work. We are also grateful to two anonymous referees whose careful reading and thoughtful comments led to a substantial improvement of the presentation. 
\bibliographystyle{habbrv} 
\bibliography{refs}
\end{document}